\documentclass[a4paper,11pt,reqno]{amsart}
\usepackage{graphics}\usepackage{pstricks}
\usepackage{amsmath,amssymb,amsthm,graphicx,verbatim,psfrag,mathdots}
\usepackage[colorlinks=true]{hyperref}

\graphicspath{{figs/}}

\theoremstyle{plain}
\newtheorem{theo}{Theorem}[section]
\newtheorem{lem}[theo]{Lemma}

\theoremstyle{definition}

\newcommand{\w}{\operatorname{W}}

\renewcommand{\r}{\operatorname{R}}

\newcommand{\sgn}{\operatorname{sgn}}

\newcommand{\sym}{{\operatorname{\mathbf{Sym}}}}
\newcommand{\asym}{{\operatorname{\mathbf{ASym}}}}

\newcommand{\e}{\operatorname{E}}
\newcommand{\fd}{\overline{\Delta}}
\newcommand{\bd}{\underline{\Delta}}
\newcommand{\id}{\operatorname{Id}}

\newcommand{\ct}{\operatorname{CT}}
\renewcommand{\r}{\operatorname{r}}
\newcommand{\p}{\operatorname{p}}
\newcommand{\q}{\operatorname{q}}
\renewcommand{\subset}{\operatorname{\mathbf{Subsets}}}

\renewcommand{\nu}{\operatorname{nu}}

\usepackage{tikz}
\usepackage{ifthen}
\usetikzlibrary{arrows,decorations.markings}

\usepackage[utf8]{inputenc}
\newcommand{\x}{0.24}
\newcommand{\s}{0.9}
\newcommand{\wbulk}[4]{W(\raisebox{-0.35em}{\begin{tikzpicture}[
    decoration={
      markings,
      mark=at position 1 with {\arrow[scale=\s,black]{latex}};    }]
    \ifthenelse{\equal{#1}{1}}{\draw [postaction={decorate}]  (0,0) -- (\x,0);} {\draw [postaction={decorate}](\x,0) -- (0,0);}
    \ifthenelse{\equal{#2}{1}}{\draw [postaction={decorate}]  (\x,-\x) -- (\x,0);} {\draw [postaction={decorate}](\x,0) -- (\x,-\x);}
    \ifthenelse{\equal{#3}{1}}{\draw [postaction={decorate}]  (\x,0) -- (2*\x,0);} {\draw [postaction={decorate}](2*\x,0) -- (\x,0);}
    \ifthenelse{\equal{#4}{1}}{\draw [postaction={decorate}]  (\x,0) -- (\x,\x);} {\draw [postaction={decorate}](\x,\x) -- (\x,0);}
  \end{tikzpicture}}}
\newcommand{\wleft}[2]{W({\,\begin{tikzpicture}[
    decoration={
      markings,
      mark=at position 1 with {\arrow[scale=\s,black]{latex}};    }]
    \ifthenelse{\equal{#1}{1}}{\draw [postaction={decorate}]  (\x,0) -- (\x,\x);} {\draw [postaction={decorate}](\x,\x) -- (\x,0);}
    \ifthenelse{\equal{#2}{1}}{\draw [postaction={decorate}]  (\x,0) -- (2*\x,0);} {\draw [postaction={decorate}](2*\x,0) -- (\x,0);}
  \end{tikzpicture}}}
\newcommand{\wright}[2]{W({\begin{tikzpicture}[
    decoration={
      markings,
      mark=at position 1 with {\arrow[scale=\s,black]{latex}};    }]
    \ifthenelse{\equal{#1}{1}}{\draw [postaction={decorate}]  (0,0) -- (\x,0);} {\draw [postaction={decorate}](\x,0) -- (0,0);}
    \ifthenelse{\equal{#2}{1}}{\draw [postaction={decorate}]  (\x,0) -- (\x,\x);} {\draw [postaction={decorate}](\x,\x) -- (\x,0);}
  \end{tikzpicture}\,}}
\newcommand{\harrow}[1]{\begin{tikzpicture}[
    decoration={
      markings,
      mark=at position 1 with {\arrow[scale=\s,black]{latex}};    }]
    \ifthenelse{\equal{#1}{1}}{\draw [postaction={decorate}]  (0,0) -- (\x,0);} {\draw [postaction={decorate}](\x,0) -- (0,0);}
  \end{tikzpicture}}

\renewcommand{\w}{\operatorname{W}}

\numberwithin{equation}{section}

\begin{document}

\title[Alternating Sign Trapezoids]{Alternating sign trapezoids and a constant term approach}

\author[Ilse Fischer]{Ilse Fischer}
\address{Ilse Fischer, Fakult\"{a}t f\"{u}r Mathematik, Universit\"{a}t Wien, Oskar-Morgenstern-Platz 1, 1090 Wien, Austria}
\email{ilse.fischer@univie.ac.at}

\thanks{The author acknowledges support from the Austrian Science Foundation FWF, START grant Y463 and SFB grant F50.}

\begin{abstract}
We show that there is the same number of $(n,l)$-alternating sign trapezoids as there is of column strict shifted plane partitions of class $l-1$ with at most $n$ parts in the top row, thereby proving a result that was conjectured independently by Behrend and Aigner. The first objects generalize alternating sign triangles, which have recently been introduced by Ayyer, Behrend and the author who showed that they are counted by the same formula as alternating sign matrices. Column strict shifted plane partitions of a fixed class were introduced in a slightly different form by Andrews, and they essentially generalize descending plane partitions. They also correspond to cyclically symmetric lozenge tilings of a hexagon with a triangular hole in the center. In addition, we also provide three statistics on each class of objects and show that their joint distribution is the same. We prove our result by employing a constant term approach that is based on the author's operator formula for monotone triangles. This paper complements a forthcoming paper of Behrend and the author, where the six-vertex model approach is used to show equinumeracy as well as generalizations involving statistics that are different from those considered in the present paper.
\end{abstract}

\maketitle

\section{Introduction}

Revealing the mysterious relation between \emph{alternating sign matrices} and \emph{plane partitions} is a fascinating and challenging project in enumerative combinatorics. For more than 35 years now, combinatorialists fail to find bijections between three classes of objects that are all counted by the following formula. 
\begin{equation}
\label{ASM}
\prod_{i=0}^{n-1} \frac{(3i+1)!}{(n+i)!}
\end{equation}
These objects are $n \times n$ alternating sign matrices (introduced in \cite{RobRum} and enumerated in  \cite{ZeilbergerASMProof, KuperbergASMProof}), \emph{totally symmetric self-complementary plane partitions} in a $2n \times 2n \times 2n$ box (introduced in \cite{StanleySym}, enumerated in \cite{TSSCPP} and related to alternating sign matrices in \cite{MilRobRum86}), and \emph{descending plane partitions} with parts less than or equal to $n$ (introduced and enumerated in \cite{AndrewsMacdonald}, and related to alternating sign matrices in \cite{DPPMRR}). Very recently, a fourth class of objects has been added to this list, namely \emph{alternating sign triangles} \cite{EXTREMEDASASM}.

In the present paper, we provide new mysteries but also new insights concerning this topic by studying the relation between generalizations of alternating sign triangles and generalizations of descending plane partitions. We use a constant term approach that was developed by the author in \cite{FischerNumberOfMT,FischerRefEnumASM,CTAST} to show ``equinumeracy'' of these two classes, which has been conjectured independently first by Behrend and then  by Aigner. The proof is again non-bijectively. However, we introduce three statistics on each family of objects and show that they have the same joint distribution, which should hint at a bijection.  This work complements a forthcoming paper by Behrend and the author \cite{trapezoids}, where equinumeracy is shown using the six-vertex model approach. In that paper, several statistics that are different from ours are provided on each of the two families of objects and it is shown that they also have the same joint distribution.

The objects generalizing alternating sign triangles are defined as follows: For 
$n \ge 1$ and $l \ge 2$, an \emph{$(n,l)$-alternating sign trapezoid} is an array of $1$'s, $-1$'s and $0$'s with $n$ centered rows of lengths $2n+l-2,2n+l-4,\ldots,l$ arranged as follows
$$
\begin{array}{cccccccccc}
a_{1,1} & a_{1,2} & a_{1,3} & \ldots & \ldots & \ldots & \ldots  & \ldots & \ldots & a_{1,2n+l-2} \\
            & a_{2,2} & a_{2,3} & \ldots  & \ldots & \ldots & \ldots & \ldots & a_{2,2n+l-3} & \\
            &              & \ddots  &           &           &           &          & \iddots &                   & \\
            &              &          & a_{n,n} & \ldots & \ldots  & a_{n,n+l-1} &     &                   &  
\end{array},
$$           
such that the following conditions are satisfied.
\begin{enumerate}
\item In each row and column, the non-zero entries alternate.
\item All row sums are $1$.
\item The topmost non-zero entry in each column is $1$ (if such an entry exists at all).
\item The column sums are $0$ for the middle $l-2$ columns, that is for the columns in positions $n+1,n+2,\ldots,n+l-2$.
\end{enumerate}
Here is a $(5,4)$-alternating sign trapezoid.
$$
\begin{array}{cccccccccccccc}
 0& 0& 0& 0& 0& 0& 0& 1& 0& 0& 0& 0 \\
   & 0& 0& 0& 0& 1& 0&-1& 1& 0& 0&  \\
   &   &  0& 1& 0&-1& 0& 1&-1& 1&   & \\
   &   &    & 0& 0& 0& 1&-1& 1&   &   & \\
   &   &    &    & 1& 0&-1& 1&  &   &   &     
\end{array}
$$   
Alternating sign triangles of order $n+1$ as defined in \cite[Definition~2.1]{EXTREMEDASASM} are in simple bijective correspondence with $(n,3)$-alternating sign trapezoids: in order to obtain an alternating sign trapezoid delete the bottom row of the alternating sign triangle (which is $1$ in any case). Alternating sign trapezoids are the infinite family announced in \cite[Section 6.5]{EXTREMEDASASM} that is in a sense a common generalization of alternating sign triangles and of quasi alternating sign triangles. The relation to the latter is the subject of Lemma \ref{qast}: it turns out that quasi alternating sign triangles with $n$ rows should really be seen as $(n,1)$-alternating sign trapezoids (they are defined as $(n,l)$-alternating sign triangles in the excluded case $l=1$ except that the row sum can be $0$ or $1$ in the bottom row) and this is the definition we will use in the following for $(n,1)$-alternating sign trapezoids.
Alternating sign trapezoids with odd base $l \ge 3$ have been introduced before by Aigner in \cite{Aigner}.

Next we define three statistics on alternating sign trapezoids if $l \ge 2$: By property (3) in the definition, the entries in each column sum either to $0$ or $1$. We refer to columns with sum $1$ as $1$-columns. As each of the $n$ rows sum to $1$, the sum of all entries in an $(n,l)$-alternating sign trapezoid is $n$ and so there are precisely $n$ $1$-columns. By property (4), these $1$-columns must be among the union of the $n$ leftmost columns and the $n$ rightmost columns. For a given alternating trapezoid $T$, we define 
$$
\r(T) = \# \text{ of $1$-columns among the $n$ leftmost columns.} 
$$
We say that a $1$-column is a $10$-column if its bottom entry is $0$.
We define 
\begin{align*}
\p(T) &= \# \text{ of $10$-columns among the $n$ leftmost columns}, \\
\q(T) &= \# \text{ of $10$-columns among the $n$ rightmost columns}.
\end{align*}
In the example above, we have $\r(T)=2, \p(T)=1, \q(T)=0$. The weight of an alternating sign trapezoid $T$ is then defined as follows.
$$
\w(T) = P^{\p(T)} Q^{\q(T)} R^{\r(T)}
$$
The weight of an $(n,1)$-alternating sign trapezoid is the following 
\begin{multline}
\label{quasiweight}
\w(T) = P^{\# \text{ of $10$-columns among the $n-1$ leftmost columns}} \, 
 Q^{\# \text{ of $10$-columns among the $n-1$ rightmost columns}} \\
\times (P+Q-1)^{[\text{the central column is a $10$-column}]} 
R^{\# \text{ of $1$-columns among the $n$ leftmost columns}}, \\
\end{multline}
where we use the Iverson bracket
$$
[\text{statement}] = \begin{cases} 1 & \text{if $\text{statement}$ is true,}\\ 0  & \text{otherwise}. \end{cases}
$$
This is in a sense the limit $l \to 1$ of the definition of the weight for $l \ge 2$ because in the limit the central column would contribute to $\p(T)$ and to $\q(T)$ if it was a $01$-column, and so it makes sense that it actually contributes $P+Q-1$. 

In order to define the objects that generalize descending plane partitions, recall that a \emph{strict partition} is a sequence $\lambda=(\lambda_1,\ldots,\lambda_n)$ of positive integers with $\lambda_1 > \ldots > \lambda_n$, and that the shifted Ferrers diagram of shape $\lambda$ is an array of cells with $n$ rows where each row is indented by one cell to the right with respect to the previous row and with $\lambda_i$ cells in row $i$. The shifted Ferrers diagram of the strict partition $(5,4,2,1)$ is displayed in Figure~\ref{shifted}. 
\begin{figure}
\scalebox{0.6}{
\includegraphics{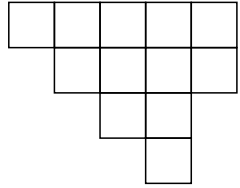}}
\caption{\label{shifted} Shifted Ferrers diagram of shape $(5,4,2,1)$.}
\end{figure}
A \emph{column strict shifted plane partition} is a filling of a shifted Ferrers diagram with positive integers such that rows are decreasing and columns are strictly decreasing. A column strict shifted plane partition of shape $(5,4,2,1)$ is for instance 
$$
\begin{array}{ccccc}
6 & 5 & 5 & 4 & 2 \\
  & 4 & 3 & 3 & 1 \\
  &   & 2 & 2 &   \\
  &   &   & 1  &  
\end{array}.
$$
A column strict shifted plane partition is of \emph{class $k$} if the first part of each row exceeds the length of the row by precisely $k$. For the example given above there is of course no $k$ such that the plane partition is of class $k$. A column strict shifted plane partition of class $2$ is for instance 
$$
\begin{array}{ccccc}
7 & 7 & 6 & 6 & 3 \\
  & 6 & 5 & 5 & 1 \\
  &   & 4 & 2 &    
\end{array}.
$$
These objects have appeared before, for the first time in slightly different form in \cite{AndrewsMacdonald}. Descending plane partitions are the objects defined in that paper that correspond to column strict shifted plane partitions of class $2$. The definition used here is taken from \cite{MRR87}. It was established in \cite{DPPKratt,55} that these objects are equivalent to cyclically symmetric lozenge tilings of a hexagon where a triangle has been removed from the center. 

Now we define the three statistics on column strict shifted plane partitions $C$. One of these statistics depends on a fixed parameter $d \in \{1,\ldots,k\}$ and so this statistic is actually a member of a family of $k$ statistics. (However, the family is empty if $k=0$ and we will provide a different definition of the weight if $d=0$.)
\begin{align*}
\r(C) &=  \# \text{ of rows}  \\
\p_d(C) &= \# \text{ of parts equal to $j-i+d$ where $i$ is the row and $j$ is the column of the part} \\
\q(C) &= \# \text{ of $1$'s}
\end{align*}
In the example above, we have $\r(C)=3, \p_{1}(C)=1, \q(C)=1$.
The weight of a column strict shifted plane partition $C$ is then 
$$
\w_d(C)=P^{\p_d(C)} Q^{\q(C)} R^{\r(C)},
$$
which clearly also depends on $d$.
For $d=0$, the weight is defined as follows\footnote{In order to see that this is definition is natural as it is somewhat the limit $d \to 0$ of the definition for $d \ge 1$, observe that, for $d \ge 1$, the first part of a row cannot contribute to $\q(C)$ (as $k \ge 1$ and the first part of a row is at least $k+1$) and that no $1$ can contribute to $\p_d(C)$ (since such a part would have to be the first part of its row, but this part is at least $k+1 \ge 2$).  Now, in the limit $d \to 0$ (excluding now first parts from $\q(C)$), a $1$ that is the second part of a row would contribute to both $\p_d(C)$ and $\q(C)$, which is why it actually makes sense that it contributes $P+Q-1$, compare also to the similar situation for $(n,1)$-alternating sign trapezoids. Also note that such a $1$ can only appear in the bottom row.}: 
\begin{multline*}
\w_0(C)= P^{\# \text{ of parts $>1$ equal to $j-i$ where $i$ is the row and $j$ is the column of the part}} \\
\times Q^{\# \text{ of $1$'s not in the first nor in the second position of a row}} \\
\times (P+Q-1)^{[\text{$1$ is the second part in the bottom row}]} \,
R^{\# \text{ of rows}}.
\end{multline*}
We are now in the position to state our main result.
\begin{theo} 
\label{main}
Let $n,l \ge 1$ and $0 \le d \le l-1$. Then the generating function of $(n,l)$-alternating sign trapezoids $T$ with respect to the weight $\w(T)$ is equal to the generating function of column strict shifted plane partitions $C$ of class $l-1$ where the length of the first row does not exceed $n$ with respect to the weight $\w_d(C)$.
\end{theo}

For $d=l-1$, the result was conjectured independently by Behrend.

\smallskip

If $l \ge 2$ and $d \ge 1$, the statement is equivalent to the fact that, for non-negative integers $p,q,r$, the number of $(n,l)$-alternating sign trapezoids $T$ with $\p(T)=p, \q(T)=q$ and $\r(T)=r$ is equal to the number of column strict shifted plane partitions $C$ of class $l-1$ where the lengths of the first row does not exceed $n$, $\p_d(C)=p, \q(C)=q$ and $\r(C)=r$. It is fairly straightforward to construct a bijection if $r=1$, which is left to the reader. The case $r>1$ seems to be less clear. 

In order to illustrate the result, we consider the case $n=2$, $l=4$. The list of $(2,4)$-alternating sign trapezoids $T$ is given next, together with the values of the statistics which are displayed below each trapezoid in the form $(\p(T),\q(T),\r(T))$.

\medskip
$$
\begin{array}{cccc}
\begin{array}{cccccc} 1&0&0&0&0&0 \\ &1&0&0&0& \end{array} &
\begin{array}{cccccc} 1&0&0&0&0&0 \\ &0&0&0&1& \end{array} &
\begin{array}{cccccc} 0&1&0&0&0&0 \\ &0&0&0&1& \end{array} &
\begin{array}{cccccc} 0&0&1&0&0&0 \\ &1&-1&0&1& \end{array} \\
(0,0,2) &  (0,0,1) & (1,0,1) & (0,0,1)  \vspace{2mm} \\ 
\begin{array}{cccccc} 0&0&0&1&0&0 \\ &1&0&-1&1& \end{array} &
\begin{array}{cccccc} 0&0&0&0&1&0 \\ &1&0&0&0& \end{array} &
\begin{array}{cccccc} 0&0&0&0&0&1 \\ &1&0&0&0& \end{array} &
\begin{array}{cccccc} 0&0&0&0&0&1 \\ &0&0&0&1& \end{array} \\
(0,0,1) & (0,1,1) & (0,0,1) & (0,0,0) 
\end{array}
$$

The column strict shifted plane partitions of class $l-1=3$ whose first row has at most $2$ parts 
and the corresponding values for the statistics (depending on the choice of $d$) are 
$$
\begin{array}{c||c|c|c|c|c|c|c|c}
    &  \emptyset & \begin{array}{c} 4 \end{array} & \begin{array}{cc} 5 & 1 \end{array} &  
\begin{array}{cc} 5 & 2 \end{array} & 
\begin{array}{cc} 5 & 3 \end{array} &
\begin{array}{cc} 5 & 4  \end{array} & 
\begin{array}{cc} 5 & 5  \end{array} &
\begin{array}{cc} 5 & 5 \\ & 4 \end{array} \\ \hline  
d=1 & (0,0,0) & (0,0,1) & (0,1,1) & (1,0,1) & (0,0,1) & (0,0,1) & (0,0,1) & (0,0,2) \\ \hline 
d=2 & (0,0,0) & (0,0,1) & (0,1,1) & (0,0,1) & (1,0,1) & (0,0,1) & (0,0,1) & (0,0,2) \\ \hline 
d=3 & (0,0,0) & (0,0,1) & (0,1,1) & (0,0,1) & (0,0,1) & (1,0,1) & (0,0,1) & (0,0,2) \\ \hline 
\end{array}.
$$
As a corollary of the theorem one obtains a closed product formula for the number of $(n,l)$-alternating sign trapezoids since Andrews has derived a formula for the number of column strict shifted plane partitions of a given class and restricted length of the first row in  \cite[Theorem 8]{AndrewsMacdonald}.

The rest of the paper is devoted to the proof of Theorem~\ref{main}. It is structured into six more sections. In Section~\ref{operatorSection}, we 
describe a known formula for the number of certain truncated monotone triangles and apply it to alternating sign trapezoids. This generalizes work from \cite{CTAST}. In Section~\ref{symmetrizerSec}, we use this to deduce a constant term formula for the generating function of alternating sign trapezoids involving the symmetrizer. The computation of the symmetrizer is then done in two steps. In Section~\ref{seperately}, we use a lemma from \cite{CTAST} to 
perform the first step. In Section~\ref{cauchy}, we use the Cauchy determinant to provide a determinant formula for the symmetrizer. In Section~\ref{binomialSec}, we show that the constant term of the symmetrizer can be expressed by a binomial determinant. In Section~\ref{InterpretSec}, the  Lindstr\"om-Gessel-Viennot theorem \cite{Lindstr,GesselViennot} is used to give the determinant the appropriate combinatorial interpretation in terms of column strict shifted plane partitions of a fixed class.

We conclude this section by explaining the relation of the work of the current paper to the work presented in \cite{CTAST}. In the latter paper, we considered the special case $l=3$, that is alternating sign triangles, and showed (among other things) that they are equinumerous with totally symmetric self-complementary plane partitions using also a constant term approach based on the operator formula for monotone triangles. In that paper, we were able to involve one statistic $\rho$, which is on the alternating sign triangle side a statistic that is a combination of the three statistics considered in the present paper. To be more precise, we have 
$$
\rho(T) = \r(T)-\p(T)+\q(T)+1.
$$
On totally symmetric self-complementary plane partitions, the corresponding statistic is one that had already been introduced in \cite{MilRobRum86}, where it was conjectured that it has the same distribution as the column of the unique $1$ in the top row of an alternating sign matrix, a fact that was eventually proven in \cite{FonZinn}. It turned out that the approach from \cite{CTAST} cannot be (easily) extended to alternating sign trapezoids. The approach presented in the current paper parallels the one from \cite{CTAST} only in Section~\ref{operatorSection}, and at some point in Section~\ref{seperately}.

\section{An operator formula for the generating function of alternating sign trapezoids with prescribed positions of the $1$-columns}
\label{operatorSection}

In this section we derive an operator formula for the generating function of alternating sign trapezoids with prescribed positions of the $1$-columns. For this purpose, we transform the objects into certain partial monotone triangles. This idea first appeared in \cite{FischerESI} and it was later used in 
\cite{Aigner,CTAST}.

Recall that a \emph{Gelfand-Tsetlin pattern} is a triangular array 
$(m_{i,j})_{1 \le j \le i \le n}$ of integers, where the entries are usually arranged as follows 
\begin{equation}
\label{indexing}
\begin{array}{ccccccccccccccccc}
  &   &   &   &   &   &   &   & m_{1,1} &   &   &   &   &   &   &   & \\
  &   &   &   &   &   &   & m_{2,1} &   & m_{2,2} &   &   &   &   &   &   & \\
  &   &   &   &   &   & \dots &   & \dots &   & \dots &   &   &   &   &   & \\
  &   &   &   &   & m_{n-2,1} &   & \dots &   & \dots &   & m_{n-2,n-2} &   &   &   &   & \\
  &   &   &   & m_{n-1,1} &   & m_{n-1,2} &  &   \dots &   & \dots   &  & m_{n-1,n-1}  &   &   &   & \\
  &   &   & m_{n,1} &   & m_{n,2} &   & m_{n,3} &   & \dots &   & \dots &   & m_{n,n} &   &   &
\end{array}
\end{equation}
such that there is a weak increase in northeast and southeast direction, i.e., 
$m_{i+1,j} \le m_{i,j} \le m_{i+1,j+1}$ for all $i,j$ with $1 \le j \le i < n$. 
A Gelfand-Tsetlin pattern in which each row is strictly increasing except for possibly the bottom row is said to be a \emph{monotone triangle}. (This definition deviates from the standard definition where also the bottom row needs to be strictly increasing.)

\begin{figure}
\scalebox{0.3}{
\includegraphics{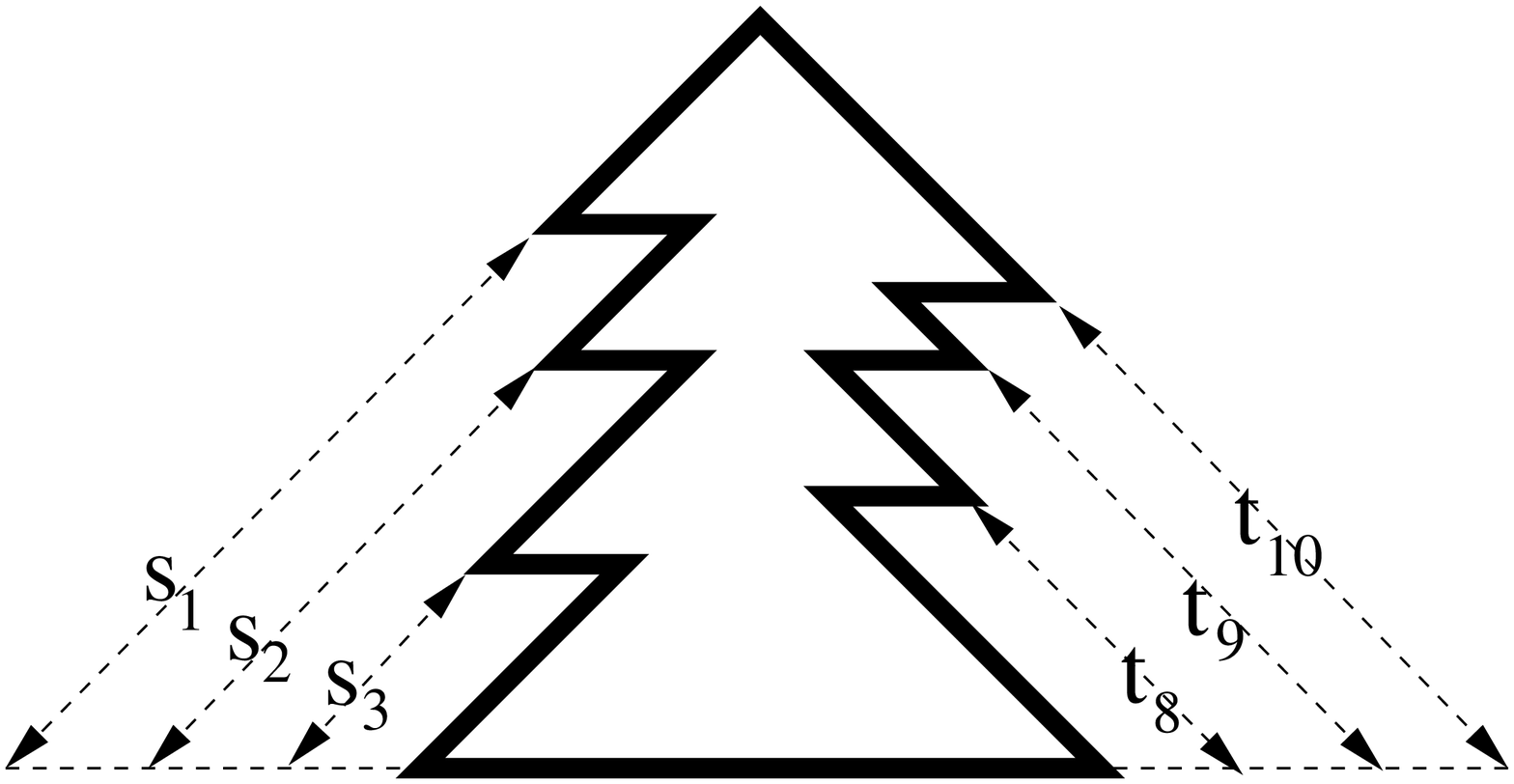}}
\caption{\label{st_tree} The shape of an $(\mathbf{s},\mathbf{t})$-tree.}
\end{figure}

We define certain partial monotone triangles which we call \emph{$(\mathbf{s},\mathbf{t})$-trees}: Let $l,r,n$ be non-negative integers with $l+r \le n$. Suppose 
$\mathbf{s}=(s_1,s_2,\ldots,s_l)$, $\mathbf{t}=(t_{n-r+1},t_{n-r+2},\ldots,t_n)$ are sequences of non-negative integers, where $\mathbf{s}$ is weakly decreasing, while $\mathbf{t}$ is weakly increasing. The shape of an $(\mathbf{s},\mathbf{t})$-tree of order $n$ is obtained from the shape of a monotone triangle with $n$ rows when deleting the bottom $s_i$ entries from the $i$-th NE-diagonal for $1 \le i \le l$  (NE-diagonals are counted from the left) and the bottom $t_i$ entries from the $i$-th SE-diagonal for $n-r+1 \le i \le n$ (SE-diagonals are also counted from the left), see Figure~\ref{st_tree}. We assume in the following that 
there is no interference between the deletion of the entries in the $l$  
leftmost NE-diagonals (as prescribed by $\mathbf{s}$) with the deletion of the entries from the $r$ rightmost SE-diagonals (as prescribed by $\mathbf{t}$).

In an $(\mathbf{s},\mathbf{t})$-tree, an entry $m_{i,j}$ is said to be \emph{regular} if it has a SW neighbour $m_{i+1,j}$ and a SE neighbour $m_{i+1,j+1}$.
We require the following monotonicity properties in an 
$(\mathbf{s},\mathbf{t})$-tree: 
\begin{enumerate}
\item Each regular entry $m_{i,j}$ has to fulfill $m_{i+1,j} \le m_{i,j} \le m_{i+1,j+1}$.
\item Two adjacent regular entries $m_{i,j}, m_{i,j+1}$ in the same row have to be distinct.   
\end{enumerate}
This extends the notion of monotone triangles, as a monotone triangle of order $n$ is just an $(\mathbf{s},\mathbf{t})$-tree, where $l,r$ are any two numbers 
with $l+r \le n$, $\mathbf{s}=(\underbrace{0,\ldots,0}_{l})$ and $\mathbf{t}=(\underbrace{0,\ldots,0}_r)$.

We fix some notation that is needed in the following: We use the \emph{shift operator} $\e_x$, the \emph{forward difference} $\fd_{x}$ and the \emph{backward difference} $\bd_{x}$, which are defined as follows.
\begin{align*}
\e_x p(x) &:= p(x+1) \\
\fd_{x} &:= \e_x - \id \\
\bd_{x} &:= \id - \e_x^{-1}
\end{align*} 
The following polynomial in $\mathbf{x}=(x_1,\ldots,x_n)$ plays an important role.
\begin{equation}
\label{operator}
M_n(\mathbf{x}):=\prod_{1 \le p < q \le n} \left (1 + \fd_{x_q} + \fd_{x_p} \fd_{x_q} \right)
\prod_{1 \le i < j \le n} \frac{x_j-x_i}{j-i}
\end{equation}
The next theorem is the main result of \cite{FischerRefEnumASM} and is based on previous work in \cite{FischerNumberOfMT}.
\begin{theo} 
\label{truncated}
Let $n,l,r$ be non-negative integers with $l+r \le n$. Suppose $b_1,\ldots,b_n$ is a weakly increasing sequence of integers, and $\mathbf{s}=(s_1,s_2,\ldots,s_l)$, $\mathbf{t}=(t_{n-r+1},t_{n-r+2},\ldots,t_n)$ are a
weakly decreasing and a weakly increasing sequence of non-negative integers, respectively.
Then the evaluation of the following polynomial
\begin{equation}
\label{stoperator}
(-\fd_{x_1})^{s_1} \cdots (-\fd_{x_l})^{s_l} 
\bd_{x_{n-r+1}}^{t_{n-r+1}} 
  \cdots \bd_{x_{n}}^{t_{n}}  
 M_n(\mathbf{x})
\end{equation}
at $(x_1,\ldots,x_n)=(b_1,\ldots,b_n)$ 
is the number of $(\mathbf{s},\mathbf{t})$-trees of order $n$ with the following properties:
\begin{itemize}
\item For $1 \le i \le n-r$, the  bottom entry of the $i$-th NE-diagonal is $b_i$.
\item For $n-r+1 \le i \le n$, the bottom entry of the $i$-th SE-diagonal is $b_i$.
\end{itemize}
\end{theo}

We aim at applying this theorem to $(n,l)$-alternating sign trapezoids with $l \ge 2$. We index the $n$ leftmost columns of an $(n,l)$-alternating sign trapezoid with $-n,-n+1,\ldots,-1$ from left to right, while the $n$ rightmost columns are indexed with $1,2,\ldots,n$. Now suppose $-n \le j_1 < j_2 < \ldots < j_m < 0 < j_{m+1} < j_{m+2} < \ldots < j_n \le n$. Next we illustrate with the help of the $(5,4)$-alternating sign trapezoid given in the introduction why $(n,l)$-alternating sign trapezoids that have their $1$-columns in positions $j_1,\ldots,j_n$ correspond to $(\mathbf{s},\mathbf{t})$-trees of order $n$ with
\begin{align*}
\mathbf{s} &= (-j_1-1,\ldots,-j_m-1) \\
\mathbf{t} &= (j_{m+1}-1,\ldots,j_n-1)
\end{align*}
such that $j_i$ is the bottom entry of the $i$-th NE-diagonal for $1 \le i \le m$ and $j_i+l-3$ is the bottom entry of the $i$-th SE-diagonal for $m+1 \le i \le n$.
The argument has essentially already appeared in \cite{FischerESI,CTAST} (where we considered the special case $l=3$) and \cite{Aigner} (where it was explained for odd base $l \ge 3$).  

Note that we have $(j_1,\ldots,j_n)=(-2,-1,1,2,3)$ in our example. First we add to each entry all the entries that are situated in the same column above the entry and obtain the following trapezoid of $1$'s and $0$'s.
$$
\begin{array}{cccccccccccccc}
 0& 0& 0& 0& 0& 0& 0& 1& 0& 0& 0& 0 \\
   & 0& 0& 0& 0& 1& 0&0& 1& 0& 0&  \\
   &   &  0& 1& 0&0& 0& 1&0& 1&   & \\
   &   &    & 1& 0& 0& 1&0& 1&   &   & \\
   &   &    &    & 1& 0&0& 1&  &   &   &     
\end{array}.
$$  
Next we record row by row the columns of the $1$'s, where only for this purpose the columns are now indexed from left to right with
$-n,-n+1,\ldots,n+l-3$, as follows (for details we refer to 
\cite[Theorem~5]{CTAST}; for the $j_i$'s we still use the indexing that was introduced before):
$$
\begin{array}{cccccc}
           &  &&    2 &&  \\
         &  & 0 & & 3 &  \\
     & -2 & & 2  && 4   \\
    -2   &&1 &&  3 &  \\
     & -1&& 2  &&
\end{array}     
$$
This is an $(\mathbf{s},\mathbf{t})$-tree with 
\begin{align*}
\mathbf{s} &= (1,0)=(-j_1-1,-j_2-1) \\
\mathbf{t} &= (0,1,2)=(j_3-1,j_4-1,j_5-1),
\end{align*}
where $j_i$ is the bottom entry of the $i$-th NE-diagonal for $1 \le i \le m=2$ and $j_i+l-3=j_i+1$ is bottom entry of the $i$-th SE-diagonal for $3=m+1 \le i \le n=5$ (the first SE-diagonal actually being empty because of the truncation to the NE-diagnals). 

It follows from Theorem~\ref{truncated} that the number of $(n,l)$-alternating sign trapezoids that have their $1$-columns in positions $j_1,\ldots,j_n$ (recall that the $n$ leftmost columns are indexed with $-n,-n+1,\ldots,-1$ and the $n$ rightmost columns are indexed with $1,2,\ldots,n$) is 
\begin{equation}
\label{operator}
\left. \left( \prod_{i=1}^{m} (-\fd_{x_i})^{-j_i-1} \prod_{i=m+1}^{n} \bd_{x_i}^{j_i-1}  \right) M_n(\mathbf{x}) \right|_{\mathbf{x}=(j_1,\ldots,j_m,j_{m+1}+l-3,\ldots,j_n+l-3)}.
\end{equation}
As $M_n(\mathbf{x})$ is a polynomial in each $x_i$, it follows that the expression is a polynomial in $l$, see also \cite{Aigner}. Interestingly, \eqref{operator} provides the values of this polynomial even if $l<2$.
Summing over all possible sequences $j_1,\ldots,j_n$ and all $m \in \{0,1,\ldots,n\}$, we obtain a polynomial $t_n(l)$ that gives the number of $(n,l)$-alternating sign trapezoids if $l \ge 2$. The case $l=1$ is also interesting and has been considered before.

 \begin{lem} 
 \label{qast}
 Let $n$ be a positive integer. Then $t_n(1)$ is the number of 
quasi alternating sign triangles of order $n$ as defined in \cite[Definition~2.3]{EXTREMEDASASM}.
\end{lem}

\begin{proof} This follows essentially by translating quasi alternating sign triangles into $(\mathbf{s},\mathbf{t})$-trees in a manner that is analogous to as it was done for alternating sign trapezoids above. More specifically, the case $j_m<-1, j_{m+1}=1$ corresponds to the case that the central column is not a $1$-column (in which case the bottom entry is of course $0$), the case $j_m=-1, j_{m+1}>1$ corresponds to the case 
that the central column is a $1$-column and the bottom entry is $1$, while the case $j_m=-1, j_{m+1} =1$ corresponds to the case that the central column is a $1$-column and the bottom entry is $0$. If $j_m<-1$ and $j_{m+1}>1$, then 
$$
\left( \prod_{i=1}^{m} (-\fd_{x_i})^{-j_i-1} \prod_{i=m+1}^{n} \bd_{x_i}^{j_i-1}\right) M_n(\mathbf{x})=0
$$
as already 
$$
\left( \prod_{i=1}^{m} (-\fd_{x_i}) \prod_{i=m+1}^{n} \bd_{x_i} \right) M_n(\mathbf{x}) = (-1)^{m} \prod_{i=1}^{m} \e_{x_i} \prod_{i=1}^{n} \bd_{x_i} 
M_n(\mathbf{x}) = 0
$$
which follows from the definition of $M_n(\mathbf{x})$ and $\prod\limits_{i=1}^{n} \bd_{x_i} \prod\limits_{1 \le i < j \le n} \frac{x_j-x_i}{j-i}=0$.
\end{proof} 

In \cite[Theorem~1.3]{EXTREMEDASASM} it was shown that there is same number of quasi alternating sign triangles with $n$ rows as there is of cyclically symmetric plane partitions in an $n \times n \times n$ box, where Andrews \cite{AndrewsMacdonald} has proven that the latter are enumerated by a simple product formula.

The enumeration can be refined as follows: The generating function of $(n,l)$-alternating sign trapezoids with respect to the weight
\begin{equation}
\label{PQ}
P^{\text{$\#$ of $10$-columns left of the center}} Q^{\text{$\#$ of $10$-columns right of the center}}
\end{equation}
is 
\begin{equation}
\label{operatorPQ}
\left. \left( \prod_{i=1}^{m} \e_{x_i} (1- P \bd_{x_i}) (-\fd_{x_i})^{-j_i-1} \prod_{i=m+1}^{n} 
\e_{x_i}^{-1} (1+ Q \fd_{x_i}) \bd_{x_i}^{j_i-1}  \right) M_n(\mathbf{x}) \right|_{\mathbf{x}=(j_1,\ldots,j_m,j_{m+1}+l-3,\ldots,j_n+l-3)}
\end{equation}
if $l \ge 2$. For alternating sign triangles, this is explained in the proof of Theorem~5 in \cite{CTAST} and the argument can be extended easily to alternating sign trapezoids.

It can be checked that the weight in \eqref{PQ} has to be replaced by 
\begin{multline*}
P^{\# \text{ of $10$-columns among the $n-1$ leftmost columns}} \, 
 Q^{\# \text{ of $10$-columns among the $n-1$ rightmost columns}} \\
\times (P+Q-1)^{[\text{the central column is a $10$-column}]} 
\end{multline*}
if $l=1$ and this implies \eqref{quasiweight}.

\section{A symmetrizer formula for the generating function of alternating sign trapezoids with prescribed positions of the $1$-columns}
\label{symmetrizerSec}

In the following, it will be more convenient to work with an alternative formula for 
$M_n(\mathbf{x})$ (and not with the one provided in \eqref{operator}). In order to state it, we need to define the \emph{symmetrizer} $\sym$ and the \emph{antisymmetrizer} $\asym$
of a function $F(Y_1,\ldots,Y_n)$ with respect to the variables $Y_1,\ldots,Y_n$:
\begin{align*}
\sym_{Y_1,\ldots,Y_n} F(Y_1,\ldots,Y_n) &= 
\sum_{\sigma \in {\mathcal S}_n} F(Y_{\sigma(1)},\ldots,Y_{\sigma(n)}) \\
\asym_{Y_1,\ldots,Y_n} F(Y_1,\ldots,Y_n) &= 
\sum_{\sigma \in {\mathcal S}_n} \sgn \sigma F(Y_{\sigma(1)},\ldots,Y_{\sigma(n)})
\end{align*}
Moreover, the constant term of a formal Laurent series $F(Y_1,\ldots,Y_n)$ is denoted by $\ct_{Y_1,\ldots,Y_n} F(Y_1,\ldots,Y_n)$.
Then, as proved in \cite[Prop. 10.1]{VSASMs},
\begin{equation}
\label{asymM}
\begin{aligned}
M_n(\mathbf{x}) 
&= \ct_{Y_1,\ldots,Y_n} \sym_{Y_1,\ldots,Y_n} 
\left[ \prod_{i=1}^{n} (1+Y_i)^{x_i} \prod_{1 \le i < j \le n} (1+Y_j + Y_i Y_j)  (Y_j-Y_i)^{-1} \right] \\
&= \ct_{Y_1,\ldots,Y_n} \asym_{Y_1,\ldots,Y_n} 
\left[ \prod_{i=1}^{n} (1+Y_i)^{x_i} \prod_{1 \le i < j \le n} (1+Y_j + Y_i Y_j) \right] \prod_{1 \le i < j \le n} (Y_j-Y_i)^{-1}.
\end{aligned}
\end{equation}
Note that 
$$\asym_{Y_1,\ldots,Y_n} 
\left[ \prod\limits_{i=1}^{n} (1+Y_i)^{x_i} \prod\limits_{1 \le i < j \le n} (1+Y_j + Y_i Y_j) \right] \prod_{1 \le i<j \le n} (Y_j-Y_i)^{-1}
$$ is in fact a polynomial in $Y_1,\ldots,Y_n$ since every antisymmetric polynomial in $Y_1,\ldots,Y_n$ is divisible by 
$\prod\limits_{1 \le i < j \le n} (Y_j-Y_i)$, and so taking the constant term means here that we actually simply need to evaluate at $\mathbf{Y}=0$. 

From \eqref{operator} it follows that we need to apply 
$$
 \prod_{i=1}^{m} \e_{x_i}^{j_i} (-\fd_{x_i})^{-j_i-1} \prod_{i=m+1}^{n} \e_{x_i}^{j_i+l-3} \bd_{x_i}^{j_i-1}   =  \prod_{i=1}^{m} \e_{x_i}^{-1} (-\bd_{x_i})^{-j_i-1} 
 \prod_{i=m+1}^{n} \e_{x_i}^{l-2} \fd_{x_i}^{j_i-1} 
$$ 
to $M_n(\mathbf{x})$ and then evaluate at $\mathbf{x}=0$ to obtain the number of 
$(n,l)$-alternating sign trapezoids with prescribed positions of the $1$-columns (where we index the columns as described above). 
We study the effect of $- \bd_{x_i}$ and of $\fd_{x_i}$ when applied to the rational function in \eqref{asymM} to which we need to apply the symmetrizer: the effect of $- \bd_{x_i}$ is the same as multiplying with $\frac{-Y_i}{1+Y_i}$ as 
$$
(1+Y_i)^{x_i-1} - (1+Y_i)^{x_i} = \frac{-Y_i}{1+Y_i} (1+Y_i)^{x_i}, 
$$
while the effect of $\fd_{x_i}$ is the same as multiplying with $Y_i$ as 
$$
(1+Y_i)^{x_i+1} - (1+Y_i)^{x_i} = Y_i (1+Y_i)^{x_i}.
$$ The effect of 
$\e_{x_i}$ is clearly the multiplication with $1+Y_i$, and so we need to apply the symmetrizer to 
\begin{equation}
\label{arg1}
\prod_{i=1}^m (-Y_i)^{-j_i-1} (1+Y_i)^{j_i} \prod_{i=m+1}^n Y_i^{j_i-1} (1+Y_i)^{l-2} 
\prod_{1 \le i < j \le n} (1+Y_j + Y_i Y_j) (Y_j-Y_i)^{-1}.
\end{equation}
Observe that in order to compute the generating function given in \eqref{operatorPQ}, 
one has to multiply with 
\begin{equation}
\label{PQweight}
\prod_{i=1}^{m} (1+Y_i - P Y_i) \prod_{i=m+1}^{n} \frac{1+Q Y_i}{1+Y_i}.
\end{equation}

To compute the total number of $(n,l)$-alternating sign trapezoids, we need to sum over all $-n \le j_1 < j_2 < \ldots < j_m < 0 < j_{m+1} < \ldots < j_n \le n$ and then over all $m \in \{0,1,\ldots,n\}$. However, as the number is zero if $j_1<-n$ or 
$j_n > n$ (it can be checked fairly easily that the polynomial has no constant term in this case), we do not have to take into account the lower bound nor the upper bound.

To simplify the computation, we replace for the moment $Y_i$ by $\frac{-Z_i}{1+Z_i}$ for $i \le m$. After this transformation, \eqref{arg1} is equal to 
\begin{multline}
\label{arg2}
\prod_{i=1}^m Z_i^{-j_i-1} (1+Z_i) \prod_{i=m+1}^n Y_i^{j_i-1} (1+Y_i)^{l-2} 
\prod_{1 \le i < j \le m} (1+Z_i + Z_i Z_j) (Z_i-Z_j)^{-1} \\
\times \prod_{m+1 \le i < j \le n} (1+Y_j + Y_i Y_j) (Y_j-Y_i)^{-1} 
\prod_{i=1}^{m} \prod_{j=m+1}^{n}  (1+Z_i+Y_j) (Z_i+Y_j+Z_i Y_j)^{-1}.
\end{multline}
Summing over all $j_i$ and using the following identity 
$$\sum_{0 \le i_1 < i_2 < \ldots < i_k} X_1^{i_1} X_2^{i_2} \ldots X_k^{i_k} = \prod_{i=1}^k 
X_i^{i-1} \left( 1 - \prod\limits_{j=i}^k X_j \right)^{-1},$$
we obtain 
\begin{multline}
\label{arg}
\prod_{i=1}^m Z_i^{m-i} (1+Z_i) \left(1 - \prod\limits_{j=1}^{i} Z_j \right)^{-1}
 \prod_{i=m+1}^n  Y_i^{i-m-1}(1+Y_i)^{l-2} \left(1- \prod\limits_{j=i}^m Y_j \right)^{-1} \\ \times 
\prod_{1 \le i < j \le m} (1+Z_i + Z_i Z_j) (Z_i-Z_j)^{-1} 
 \prod_{m+1 \le i < j \le n} (1+Y_j + Y_i Y_j) (Y_j-Y_i)^{-1} \\
\times \prod_{i=1}^{m} \prod_{j=m+1}^{n}  (1+Z_i+Y_j) (Z_i+Y_j+Z_i Y_j)^{-1}.
\end{multline}

\section{Applying the symmetrizer to $Y_1,\ldots,Y_{m}$ and to 
$Y_{m+1},\ldots,Y_n$ seperately}
\label{seperately}

We will compute the symmetrizer in two steps. In order to explain this, we need another definition: Suppose $F(Y_1,\ldots,Y_n)$ is a function and $0 \le m \le n$, and denote by ${\mathcal S}_{n}^m$ the subset of the set ${\mathcal S}_{n}$ of all permutations of $\{1,\ldots,n\}$  that contains the permutations with $\sigma(i)<\sigma(j)$ for all $1 \le i < j \le m$  and all $m+1 < i < j \le n$ (these permutations clearly correspond to the subsets of $\{1,2,\ldots,n\}$ with $m$ elements). Then we define 
$$
\subset_{Y_1,\ldots,Y_m}^{Y_{m+1},\ldots,Y_n}  F(Y_1,\ldots,Y_n) = \sum_{\sigma \in {\mathcal S}_{n}^m} F(Y_{\sigma(1)},\ldots,Y_{\sigma(n)}).
$$
Hence
$$
\sym_{Y_1,\dots,Y_n} F(Y_1,\ldots,Y_n) = \subset_{Y_1,\ldots,Y_m}^{Y_{m+1},\ldots,Y_n} \sym_{Y_{1},\ldots,Y_m}  \sym_{Y_{m+1},\ldots,Y_{n}} F(Y_1,\ldots,Y_n).
$$

Now we apply $\sym_{Y_1,\ldots,Y_m}$ (or, equivalently, $\sym_{Z_1,\ldots,Z_m}$) and $\sym_{Y_{m+1},\ldots,Y_n}$ to \eqref{arg}, and in order to do so we 
use the following lemma from \cite{CTAST}.
\begin{lem} Let $n \ge 1$. Then 
\begin{multline}
\label{asym}
\asym_{X_1,\ldots,X_n} \left[ \prod_{1 \le i < j \le n} (1+X_j+X_i X_j) \prod_{i=1}^{n} X_i^{i-1} \left(1-\prod_{j=i}^{n} X_j\right)^{-1} \right] \\
= \prod_{i=1}^{n} (1-X_i)^{-1} \prod_{1 \le i < j \le n} 
\frac{(1+X_i + X_j)(X_j-X_i)}{1-X_i X_j}.
\end{multline}
\end{lem}
This gives 
$$
\prod_{i=1}^m \frac{1+Z_i}{1-Z_i}
 \prod_{i=m+1}^n  \frac{(1+Y_i)^{l-2}}{1-Y_i}
\prod_{1 \le i < j \le m} \frac{1+Z_i +  Z_j}{1-Z_i Z_j} 
\prod_{m+1 \le i < j \le n} \frac{1+Y_i + Y_j}{1-Y_i Y_j} 
\prod_{i=1}^{m} \prod_{j=m+1}^{n}  \frac{1+Z_i+Y_j}{Z_i+Y_j+Z_i Y_j}.
$$
For $i \le m$, we transform back to the original variables $Y_i$, that is we replace $Z_i$ by $-\frac{Y_i}{1+Y_i}$ for $i \le m$. We obtain 
$$
\prod_{i=1}^m \frac{1}{1+2 Y_i}
 \prod_{i=m+1}^n  \frac{(1+Y_i)^{l-2}}{1-Y_i}
\prod_{1 \le i < j \le m} \frac{1-Y_i Y_j}{1+ Y_i +  Y_j} 
\prod_{m+1 \le i < j \le n} \frac{1+Y_i + Y_j}{1-Y_i Y_j} 
\prod_{i=1}^{m} \prod_{j=m+1}^{n}  \frac{1+Y_j+ Y_i Y_j}{Y_j-Y_i}.
$$

\section{Deriving a multivariate determinant}
\label{cauchy}

We need to compute the evaluation at $\mathbf{Y}=0$ of 
\begin{multline*}
 \subset_{Y_1,\ldots,Y_m}^{Y_{m+1},\ldots,Y_n}  
 \prod_{i=1}^{m} \frac{1}{1+2 Y_i} \prod_{i=m+1}^{n} \frac{(1+Y_i)^{l-2}}{1-Y_i} \\
\times \prod_{1 \le i < j \le m} \frac{1-Y_i Y_j}{1+Y_i + Y_j} \prod_{m+1 \le i < j \le n} \frac{1+Y_i+ Y_j}{1-Y_i  Y_j} \prod_{i=1}^{m} \prod_{i=m+1}^{n} \frac{1+Y_j+Y_i Y_j}{Y_j-Y_i}.
\end{multline*}
We can divide the rational function to which we need to apply the subset operator by $\prod_{1 \le i < j \le n} (1+Y_i+Y_j)(1-Y_i Y_j)$ (since it is symmetric and thus invariant under the application of $\subset_{Y_1,\ldots,Y_m}^{Y_{m+1},\ldots,Y_n}$, and has constant term $1$) and 
obtain 
\begin{multline}
\label{c1}
 \prod_{i=1}^{m} \frac{1}{1+2 Y_i} \prod_{i=m+1}^{n} \frac{(1+Y_i)^{l-2}}{1-Y_i} \\
\times \prod_{1 \le i < j \le m} \frac{1}{(1+Y_i + Y_j)^2} \prod_{m+1 \le i < j \le n} \frac{1}{(1-Y_i  Y_j)^2} \prod_{i=1}^{m} \prod_{j=m+1}^{n} \frac{1+Y_j+Y_i Y_j}{(Y_j-Y_i)(1+Y_i+Y_j)(1-Y_i Y_j)}  \\
=  \prod_{i=m+1}^{n} (1+Y_i)^{l-1}  \prod_{i,j=1}^{m}  \frac{1}{1+Y_i + Y_j} \prod_{i,j=m+1}^{n} \frac{1}{1-Y_i  Y_j} \prod_{i=1}^{m} \prod_{j=m+1}^{n} \frac{1+Y_j+Y_i Y_j}{(Y_j-Y_i)(1+Y_i+Y_j)(1-Y_i Y_j)}.
\end{multline}
Now we use the Cauchy determinant.
$$
\det_{1 \le i, j \le n} \left( \frac{1}{X_i+Y_j} \right) = \frac{\prod\limits_{1 \le i < j \le n} (X_j-X_i)(Y_j-Y_i)}{\prod\limits_{i,j=1}^{n} (X_i+Y_j)}
$$
Setting $X_i=1+Y_i$ if $i \le m$, and $X_i=-\frac{1}{Y_i}$ if $i>m$, we obtain 
\begin{multline*}
\det_{1 \le i, j \le n} \left( \frac{1}{X_i+Y_j} \right) 
= (-1)^{m-n} \prod_{i=m+1}^n Y_i  \prod_{i,j=1}^{m} \frac{1}{1+Y_i+Y_j}\prod_{i,j=m+1}^{n} \frac{1}{1- Y_i Y_j} \\
\times \prod_{i=1}^{m} \prod_{j=m+1}^{n} \frac{(Y_j-Y_i)(1+Y_j + Y_i Y_j)}{(1+Y_i+Y_j)(1-Y_i Y_j)} \prod_{1 \le i < j \le m} (Y_j-Y_i)^2 
\prod_{m+1 \le i < j \le n} (Y_j-Y_i)^2 
\end{multline*}
It follows that 
\begin{multline*}
\frac{ \det_{1 \le i, j \le n} \left( \begin{cases} \frac{1}{1+Y_i + Y_j}, & i \le m \\ \frac{(1+Y_i)^{l-1}}{1-Y_i Y_j}, & i >m \end{cases} \right)}
{\prod\limits_{1 \le i < j \le n} (Y_j - Y_i)^2} \\
=  \prod_{i=m+1}^{n} (1+Y_i)^{l-1}  \prod_{i,j=1}^{m} \frac{1}{1+Y_i+Y_j}\prod_{i,j=m+1}^{n} \frac{1}{1- Y_i Y_j} 
\prod_{i=1}^{m} \prod_{j=m+1}^{n} \frac{1+Y_j + Y_i Y_j}{(Y_j-Y_i)(1+Y_i+Y_j)(1-Y_i Y_j)}.
\end{multline*}
The right hand side is equal to the right hand side of \eqref{c1}.
The following observation is crucial: Suppose $\sigma \in {\mathcal S}_{n}^{m}$, then, by the antisymmetry of the determinant in rows and columns, 
\begin{multline*}
\sigma \left[  \prod_{i=m+1}^{n} (1+Y_i)^{l-1}  \prod_{i,j=1}^{m} \frac{1}{1+Y_i+Y_j}\prod_{i,j=m+1}^{n} \frac{1}{1- Y_i Y_j} 
\prod_{i=1}^{m} \prod_{j=m+1}^{n} \frac{1+Y_j + Y_i Y_j}{(Y_j-Y_i)(1+Y_i+Y_j)(1-Y_i Y_j)} \right] \\
= \frac{ \det\limits_{1 \le i, j \le n} \left( \begin{cases} \frac{1}{1+Y_i + Y_j}, & i \in \sigma(\{1,\ldots,m\}) \\ \frac{(1+Y_i)^{l-1}}{1-Y_i Y_j}, & i \in \sigma(\{m+1,\ldots,n\}) \end{cases} \right)}
{\prod\limits_{1 \le i < j \le n} (Y_j - Y_i)^2}.
\end{multline*}
Applying $\subset_{Y_1,\ldots,Y_m}^{Y_{m+1},\ldots,Y_n}$ and summing over all $m$ gives 
$$
\frac{ \det\limits_{1 \le i, j \le n} \left( \frac{1}{1+Y_i + Y_j}+ \frac{(1+Y_i)^{l-1}}{1-Y_i Y_j} \right)}
{\prod\limits_{1 \le i < j \le n} (Y_j - Y_i)^2},
$$
where we have used the linearity in each row.
We can also keep track of $m$ as the exponent of $R$ as follows
$$
\frac{ \det\limits_{1 \le i, j \le n} \left( R \frac{1}{1+Y_i + Y_j} +  \frac{(1+Y_i)^{l-1}}{1-Y_i Y_j} \right)}
{\prod\limits_{1 \le i < j \le n} (Y_j - Y_i)^2}.
$$
In order to involve the $P$-weight and the $Q$-weight, we have to use
$$
\frac{ \det\limits_{1 \le i, j \le n} \left( R \frac{1+Y_i - P Y_i}{1+Y_i + Y_j}+ \frac{(1+Y_i)^{l-2} (1+Q Y_i)}{1-Y_i Y_j} \right)}
{\prod\limits_{1 \le i < j \le n} (Y_j - Y_i)^2}, 
$$
see \eqref{PQweight}.

\section{Deriving a binomial determinant}
\label{binomialSec}

Let 
$$
F(X,Y) = R \frac{1+X - P X}{1+X + Y}+ \frac{(1+X)^{l-2} (1+Q X)}{1-X Y}.
$$
and interpret it as a power series in $X$ and $Y$. 
By a formula that appeared in \cite[Eq (43)-(47)]{BehrendWeightedEnum}, 
we have
$$ \left. \frac{\det_{1 \le i, j \le n} \left( F(X_i,Y_j) \right)}{\prod_{1 \le i < j \le n} (X_j-X_i)(Y_j-Y_i)} \right|_{\mathbf{X}=\mathbf{Y}=0} = 
\det_{0 \le i,j \le n-1} ( [X^i Y^j] F(X,Y) ).
$$ 
Now 
\begin{multline*}
F(X,Y) 
= R \sum_{i,j \ge 0} (-1)^{i+j}  \left(\binom{i+j}{i} + (P-1) \binom{i+j-1}{i-1} \right)  X^{i} Y^j 
 \\ + \sum_{j,q \ge 0} \left(\binom{l-2}{q} + Q \binom{l-2}{q-1} \right) X^{j+q} Y^j, \end{multline*} 
 and so, using $\binom{n}{k} = (-1)^k \binom{k-n-1}{k}$ and 
 $\binom{n}{k} = \binom{n-1}{k} + \binom{n-1}{k-1}$, 
$$
[X^i Y^j] F(X,Y) 
= R  (-1)^{j} \left( \binom{-j}{i} - P \binom{-j-1}{i-1} \right)  +  \left( \binom{l-2}{i-j} + Q \binom{l-2}{i-j-1}. 
 \right)
$$
We multiply the matrix on the left with the matrix 
$(\binom{-l+2}{i-j})_{0 \le i, j \le n-1}$ (which has determinant $1$) and obtain, using the Chu-Vandermonde summation and $\binom{n}{k} = (-1)^k \binom{k-n-1}{k}$,
\begin{multline*}
R (-1)^{j} \left( \binom{-l-j+2}{i} - P \binom{-l-j+1}{i-1} \right)  +  \delta_{i,j}    + Q \delta_{i,j+1} \\
= R (-1)^{i+j} \left( \binom{i+j+l-3}{i} + P \binom{i+j+l-3}{i-1} \right)  +  \delta_{i,j}    + Q \delta_{i,j+1}.
\end{multline*} 
The determinant is unchanged when we multiply the $j$-th column by $(-1)^j$ and the $i$-th column by $(-1)^i$, $i,j=0,1,\ldots,n-1$, and so the matrix entry can 
clearly be replaced by
$$
R \left(\binom{i+j+l-3}{i} + P \binom{i+j+l-3}{i-1} \right)   +   \delta_{i,j}    - Q \delta_{i,j+1}.
$$
Let 
$$
K(n) = (\delta_{i,j}    - Q \delta_{i,j+1})_{ 0 \le i,j \le n-1}.
$$
Then 
$$
K(n)^{-1} = \begin{cases} Q^{i-j} & \text{if $i \ge j$} \\ 
                                              0 & \text{if $i < j$} \end{cases}.
$$   
Finally, we multiply the matrix underlying the determinant on the left with $K(n)^{-1}$, and obtain 
\begin{equation}
\label{det}
\det_{0 \le i,j \le n-1} \left( R \sum_{k=0}^{i} Q^{i-k} \left( \binom{k+j+l-3}{k} + P \binom{k+j+l-3}{k-1} \right) +  \delta_{i,j} \right).
\end{equation}
For $P=Q=R=1$, the matrix entry simplifies to
$$
\binom{i+j+l-1}{i} + \delta_{i,j}.
$$  
This determinant is known to enumerate column strict shifted plane partition of class $l-1$ and with at most $n$ parts in row $1$, see \cite{AndrewsMacdonald}. This can be shown using the Lindstr\"om-Gessel-Viennot theorem \cite{Lindstr,GesselViennot}, which will also be used in the next section to interpret the more general determinant.

\section{Interpreting the determinant as the generating function of families of non-intersecting lattice paths and column strict shifted plane partitions}  
\label{InterpretSec}  

We first consider the case $P=1$. Then \eqref{det} is equal to
\begin{equation}
\label{P1}
\sum_{r=0}^{n} R^r \sum_{0 \le u_1 < u_2 < \ldots < u_{r} \le n-1} \det_{1 \le i,j \le r} \left( \sum_{k=0}^{u_i} Q^{u_i-k}  \binom{k+u_j+l-2}{k}    \right).
\end{equation}
The binomial coefficient $\binom{k+u_j+l-2}{k}$ is the number of lattice paths starting in $(k,1)$ and ending in $(0,u_j+l-1)$, allowing north unit steps and west unit steps. It follows that the matrix entry is the generating function of lattice paths starting in 
$(u_i,0)$ and ending $(0,u_j+l-1)$ with respect to the weight 
$$
Q^{\# \text{ of horizontal steps at height $0$}}.
$$
Hence, by the Lindstr\"om-Gessel-Viennot theorem \cite{Lindstr,GesselViennot}, \eqref{P1} is the generating function of 
families of non-intersecting lattice paths where the starting points are among $A_i=(i,0)$, $i=0,1,\ldots,n-1$, and the end points are among $B_j=(0,j+l-1)$, 
$j=0,1,\ldots,n-1$, such that $A_u$ is a starting point if and only if $B_u$ is an ending point, allowing north unit steps and west unit steps and where the weight is 
$$
R^{\# \text{ of paths}} Q^{\# \text{ of horizontal steps at height $0$}}.
$$ 
Such families are equivalent to column strict shifted plane partitions of class $l-1$ with at most $n$ parts in the top row: Each path corresponds to a row, and the parts of the plane partitions are just the heights of the horizontal steps increased by $1$ and where we add the increased height of the end point of the path at the beginning of the corresponding row of the plane partition.  The family of non-intersecting lattice paths that corresponds to
$$
\begin{array}{ccccc}
7 & 7 & 6 & 6 & 3 \\
  & 6 & 5 & 5 & 2 \\
  &   & 4 & 4 &   \\
  &   &   & 3  &  
\end{array}.
$$
is given in Figure \ref{CSSP}.

\begin{figure}
\scalebox{0.5}{
\psfrag{1}{\Large$1$}
\psfrag{2}{\Large$2$}
\psfrag{3}{\Large$3$}
\psfrag{4}{\Large$4$}
\psfrag{5}{\Large$5$}
\psfrag{6}{\Large$6$}
\psfrag{7}{\Large$7$}
\includegraphics{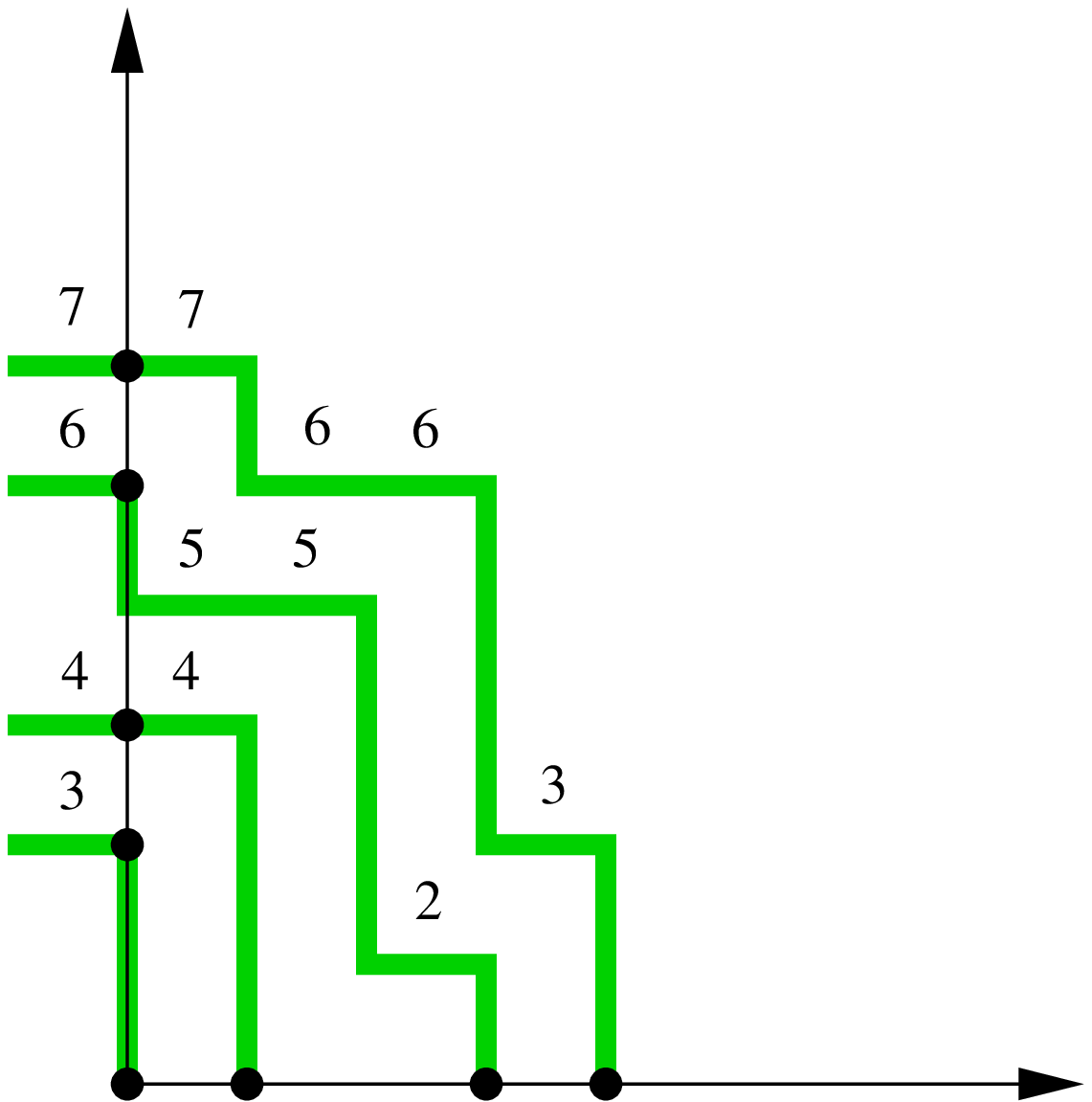}}
\caption{\label{CSSP}}
\end{figure}

We now consider the general case and assume $d \ge 1$ (which implies also $l \ge 2$). Note that 
$$
\sum_{k=0}^{i} Q^{i-k} \binom{k+j+l-3}{k} 
$$
is the generating function of lattice paths starting in $(i,0)$ and ending in $(0,j+l-1)$ with respect to the number of horizontal steps at height $0$ where the line $y=x+d$ is reached through a north step, while 
$$
\sum_{k=0}^{i} Q^{i-k} \ \binom{k+j+l-3}{k-1} 
$$  
is the generating function of lattice paths starting in $(i,0)$ and ending in $(0,j+l-1)$ where the line $y=x+d$ is reached through a west step. Employing again the Lindstr\"om-Gessel-Viennot theorem,
this concludes the proof of Theorem~\ref{main} for $d \ge 1$.

Now we assume $d=0$. Then it can be checked easily that 
$$
\sum_{k=0}^{i} Q^{i-k} \left( \binom{k+j+l-3}{k} + P \binom{k+j+l-3}{k-1} \right)
$$
is the generating function of lattice paths starting in $(i,0)$ and ending in 
$(0,j+l-1)$ with respect to the following weight.
\begin{multline*}
P^{[\text{the path reaches the main diagonal through a west step but not in the origin}]} \\
\times Q^{\# \text{ of horizontal steps on the $x$-axis not containing the origin}} \\
\times (P+Q-1)^{\text{[the path reaches the main diagonal in the origin]}}
\end{multline*}
This concludes the proof of Theorem~\ref{main} for $d=0$.

\end{document}